\documentclass[a4paper,10pt,reqno]{amsart}

\usepackage{etex}
\usepackage{amssymb}
\usepackage{amsmath}
\usepackage{amsfonts}
\usepackage{amsthm}
\usepackage{mathrsfs}
\usepackage{eucal}
\usepackage[safe]{tipa}
\usepackage[normalem]{ulem}
\usepackage{cases}
\usepackage{xr}
\usepackage{color}
\usepackage{enumitem}

\usepackage{ulem}

\usepackage{slashed}
\usepackage{textalpha}

\usepackage{lpic}

\makeatletter
\newcommand\mathcircled[1]{%
  \mathpalette\@mathcircled{#1}%
}
\newcommand\@mathcircled[2]{%
  \tikz[baseline=(math.base)] \node[draw,circle,inner sep=1pt] (math) {$\m@th#1#2$};%
}
\makeatother

\usepackage{url}
\usepackage[linktocpage=true]{hyperref}

\hypersetup{
    colorlinks=true,
    linkcolor=blue,
    filecolor=blue,      
    urlcolor=blue,
}

\newcommand{\la}{\langle}
\newcommand{\ra}{\rangle}

\newcommand{\Scal}S

\theoremstyle{plain}
\newtheorem{thm}{Theorem}[section]

\newtheorem{lemma}[thm]{Lemma}

\newtheorem{remark}[thm]{Remark}

\def\d{{\rm d}}

\title[Infinitesimal rigidity of Hermitian instantons]
{Infinitesimal rigidity of Hermitian gravitational instantons}

\author[L. Andersson]{Lars Andersson}

\address[Lars Andersson]{University of Potsdam, Department of Mathematics, Karl-Liebknecht-Str. 24-25, 14476 Potsdam, Germany}
\email{lars.andersson@uni-potsdam.de}

\author[B. Araneda]{Bernardo Araneda}
\address[Bernardo Araneda]{School of Mathematics and Maxwell Institute for Mathematical Sciences, University of Edinburgh, EH9 3FD, United Kingdom,  \newline
Max-Planck-Institut f\"ur Gravitationsphysik (Albert-Einstein-Institut), 
Am M\"uhlenberg 1, D-14476 Potsdam, Germany}
\email{baraneda@ed.ac.uk, bernardo.araneda@aei.mpg.de}

\allowdisplaybreaks[2]
\numberwithin{equation}{section}

\begin{document}

\begin{abstract}

We prove infinitesimal rigidity and integrability of the moduli space for Hermitian gravitational instantons.
Together with the recent proof by Biquard, Gauduchon, and LeBrun \cite{BGL} of local rigidity for Hermitian instantons, this completes the picture of the moduli space of Hermitian gravitational instantons, both for the compact and non-compact cases.

An important step in the proof is to show that provided certain boundary
conditions hold, a curve of Riemannian metrics passing through a Hermitian
non-Kähler Einstein metric is conformally Kähler to second perturbative order.
This uses ideas of Wu \cite{Wu} and LeBrun \cite{Lebrun21}. 

\end{abstract} 

\date{\today}
\maketitle

\section{Introduction}

A gravitational instanton is a compact, or complete and asymptotically flat,  Einstein or Ricci-flat 4-manifold. 
While the notion of gravitational instanton is sometimes taken to imply special geometry such as the half-flat condition we do not, unless otherwise stated, impose such restrictions here. With suitable fall-off conditions on curvature, the asymptotically flat instantons can be classified depending on volume growth at infinity, the ones relevant here are asymptotically locally Euclidean (ALE, with quartic volume growth) and asymptotically locally flat (ALF, with cubic volume growth), but there are examples with quadratic, linear and even non-polynomical volume growth, cf. \cite{SunZhang} and references therein. ALF instantons are asymptotic to a circle bundle over $S^2$ at infinity. ALF instantons where this circle bundle is trivial are called asymptotically flat (AF). 

Important examples of gravitational instantons include the Euclidean signature versions of the Schwarzschild and Kerr black holes, the half-flat Taub-NUT and Eguchi-Hansen spaces, Chen-Teo and the compact Page and Chen-LeBrun-Weber spaces. See \cite{Aksteiner:2023djq, Lebrun2010} and references therein. Methods for construction of new instantons include resolution of singularities, gluing, the Gibbons-Hawking method, and inverse scattering. The Gibbons-Hawking method allows one to construct large families of half-flat ALF or ALE instantons, which are in general $S^1$-symmetric. The inverse scattering, or Belinskii-Zakharov method, was used to construct the Chen-Teo instanton \cite{ChenTeo1}, an ALF instanton which is Hermitian with respect to only one orientation \cite{Aksteiner:2021fae}. Symmetry reduction of the Einstein equation leads to a harmonic map equation and this has been exploited in the work of Li and Sun \cite{LiSun}, following the approach of Kunduri and Lucietti \cite{Kunduri:2021xiv}, to construct infinite families of toric AF instantons.

Combining results from Goldberg-Sachs \cite{GoldbergSachs} and Derdzinski \cite{Derdzinski}, LeBrun showed that Hermitian instantons are conformal to extremal K\"ahler spaces \cite{Lebrun95}. See also \cite{Przanowski, ApostolovGauduchon, Gover} for more treatments of the Goldberg-Sachs theorem in Riemannian signature, and \cite[Chapters VIII, IX]{Flaherty} for the Lorentzian case. 
Using the above result, LeBrun provided a complete classification of compact Hermitian-Einstein instantons \cite{Lebrun2010}. 
Biquard and Gauduchon obtained a complete classification of toric, Hermitian non-K\"ahler ALF instantons \cite{Biquard:2021gwj}.
Recently the classification was extended by Li \cite{Li:2023jlq} to Hermitian ALF and ALE instantons. The classification of Hermitian ALE instantons is not complete, but requires some additional conditions on the structure at infinity. 
In the toric case, recent work by Lucietti and one of the authors
provides a classification of Hermitian toric ALE instantons without additional conditions at infinity \cite{AranedaLucietti}.
By the classification results for Hermitian instantons, the new examples of Li and Sun \cite{LiSun} are necessarily algebraically general. In \cite{Aksteiner:2023djq}, one of the authors and collaborators provided steps towards a classification of $S^1$-symmetric instantons. In particular, this result shows that an $S^1$-symmetric instanton with the same topology as the Kerr, Chen-Teo, or Taub-bolt spaces is necessarily in the known moduli space. In view of the results of Li and Sun, the general classification problem is wide open, even for the case of toric instantons. 

The work mentioned above gives many examples of moduli spaces of gravitational instantons. The known moduli spaces of ALF instantons on a manifold with given topology have dimension at most 2\footnote{Our definition of moduli space does not factor out the scaling, in contrast to e.g. \cite{Lebrun21}}. 

Local rigidity for  Hermitian ALF instantons has been  recently proved by Biquard, Gauduchon and LeBrun \cite{BGL}. This result is analogous to the local rigidity theorems valid for Calabi-Yau spaces \cite{DWW}. However, it is a priori possible that the moduli space of gravitational instantons fails to be integrable 
in the sense that there could exist infinitesimal Einstein deformations that are not tangent to curves in the moduli space.
It is a notable open problem whether integrability holds in general for Einstein metrics, see the discussion in \cite[§12.E]{Besse} and \cite[§2.10]{Biquard2007}.
A classical example where integrability fails was given by Koiso \cite{Koiso82}.
Thus the result on local rigidity leaves open the question of infinitesimal rigidity and integrability of Hermitian instantons.
Here, infinitesimal rigidity is the statement that any ALF infinitesimal deformation is, up to gauge, a perturbation with respect to the moduli parameters.
In \cite{Andersson:2024wtn}, we conjectured that integrability holds for Hermitian ALF instantons. 
The following theorem, which is the main result of this paper, settles this question. 

\begin{thm}\label{maintheorem-early}
Let $(M,g_{ab})$ be an ALF Hermitian, non-K\"ahler gravitational instanton. Then $(M,g_{ab})$ is integrable and infinitesimal rigidity holds.
\end{thm}

A step towards this result was to prove mode stability \cite{Andersson:2024wtn}, i.e., to prove that the self-dual part of the linearized Weyl tensor of generic ALF infinitesimal deformations is algebraically special. This suggests that the deformations stay within the same family of instantons, and in this paper we shall prove that this is indeed the case, by using a divergence identity valid for generic Riemannian manifolds with non-self-dual Weyl tensor, which we obtain in section \ref{sec:identity}. The proof of theorem \ref{maintheorem-early} is given in section \ref{sec:rigidity}.

\section{A divergence identity}
\label{sec:identity}

\subsection{Preliminaries}
We shall mainly use the abstract index notation following \cite{PR1}, although we occasionally use an index-free notation. Let $(M,g_{ab})$ be a four-dimensional, smooth, orientable Riemannian manifold, with Levi-Civita connection $\nabla_{a}$, Hodge star operator $\star$ and volume form $\varepsilon_{abcd}$. The metric induces an isomorphism between $TM$ and $T^{*}M$ by raising and lowering indices, which extends to higher-valence tensors. The inner product $\la\cdot,\cdot\ra_{g}$ is defined by
\begin{align}\label{innprod1}
\la A,B \ra_{g}:=g^{a_1b_1}...g^{a_kb_k}A_{a_1...a_{k}}B_{b_1...b_{k}},
\end{align}
for any tensor fields $A,B$ with the same number of indices. The point-wise norm of $A$ is $|A|^2_{g}=\la A,A \ra_{g}$. 
Let $\Lambda^{k}$ be the space of $k$-forms on $M$, and $\Gamma(\Lambda^{k})$ the space of smooth sections. The exterior derivative $d:\Gamma(\Lambda^{k})\to\Gamma(\Lambda^{k+1})$ and the codifferential $d^{*}=-\star d\: \star:\Gamma(\Lambda^{k})\to\Gamma(\Lambda^{k-1})$ are given by
\begin{align}
 (d\alpha)_{a_1...a_{k+1}}=(k+1)\nabla_{[a_1}\alpha_{a_2...a_{k+1}]}, 
 \quad 
(d^{*}\alpha)_{a_1...a_{k-1}} = -g^{bc}\nabla_{b}\alpha_{ca_1...a_{k-1}},
\end{align}
for any $\alpha\in\Gamma(\Lambda^{k})$. The Laplace-de Rham operator $\triangle:\Gamma(\Lambda^{k})\to\Gamma(\Lambda^{k})$ is
\begin{align}\label{LdR}
 \triangle := dd^{*}+d^{*}d.
\end{align}
Its relation to the standard Laplacian, $\Box:=g^{ab}\nabla_{a}\nabla_{b}$, is given by a Weitzenb\"ock identity. For example, for any 2-form $F_{ab}$, we have
\begin{align}\label{LdRF}
(\triangle F)_{ab} = -\Box F_{ab}-R_{ab}{}^{cd}F_{cd} - 2R_{[a}{}^{c}F_{b]c},
\end{align}
where $R_{abc}{}^{d}$ and $R_{ac}=R_{abc}{}^{b}$ are the Riemann and Ricci tensors of $g_{ab}$. 
Our convention for the Riemann tensor is $(\nabla_a\nabla_b-\nabla_b\nabla_a)v_c=R_{abc}{}^{d}v_{d}$ for any $v_a$.

The Hodge star $\star$ satisfies $\star^2=1$ when acting on $\Lambda^2$. This gives a splitting $\Lambda^2=\Lambda^2_{+}\oplus\Lambda^{2}_{-}$, where $\Lambda^{2}_{\pm}$ is the eigenspace of $\star$ with eigenvalue $\pm1$. Elements of $\Lambda^2_{+}$ are called self-dual 2-forms, and elements of $\Lambda^2_{-}$ are anti-self-dual 2-forms. Any $F\in\Lambda^2$ can be decomposed into self-dual and anti-self-dual pieces as $F=F^{+}+F^{-}$, with $F^{\pm}=\frac{1}{2}(F\pm\star F)$. The splitting $\Lambda^2=\Lambda^2_{+}\oplus\Lambda^{2}_{-}$ is orthogonal, meaning that, for any $\alpha\in\Lambda^2$ and $F^{+}\in\Lambda^{2}_{+}$, it holds
\begin{align}\label{orthsplit}
 \la F^{+},\alpha \ra_{g} = \la F^{+},\alpha^{+} \ra_{g},
\end{align}
where a superscript $+$ in $\alpha^{+}$ means ``self-dual part of $\alpha$''. We have
\begin{align}
 \triangle F^{+} ={}& 2 (dd^{*}F^{+})^{+} \label{LdRFp} 
\end{align}
for any $F^{+}\in\Gamma(\Lambda^{2}_{+})$.
To show this, note that since $d^{*}=-\star d\: \star$ and $F^{+}=\star F^{+}$, we have $d^{*}d F^{+}=-\star d\: \star d \star F^{+}=\star dd^{*}F^{+}$. Using this identity, we have $\triangle F^{+}=dd^{*}F^{+}+d^{*}d F^{+}=(1+\star)dd^{*}F^{+}=2 (dd^{*}F^{+})^{+}$, which gives \eqref{LdRFp}.

\smallskip
The above identities can be generalized to tensor-valued $k$-forms. We shall be interested in tensor fields of the form $Z_{a_1...a_{k}}{}^{m_1...m_l}$, where $Z_{a_1...a_{k}}{}^{m_1...m_l}=Z_{[a_1...a_{k}]}{}^{[m_1...m_l]}$. These objects are sections of the tensor bundles $\Lambda^{k}\otimes\Lambda^{l*}$ (where $\Lambda^{l*}$ is the dual of $\Lambda^{l}$), so they can be regarded as maps $\Gamma(\Lambda^{l})\to\Gamma(\Lambda^{k})$. We extend the action of $d$ and $d^{*}$ as
\begin{equation}\label{covED}
\begin{aligned}
 (dZ)_{a_1...a_{k+1}}{}^{m_1...m_l} ={}& (k+1)\nabla_{[a_1}Z_{a_2...a_{k+1}]}{}^{m_1...m_l}, \\
 (d^{*}Z)_{a_1...a_{k-1}}{}^{m_1...m_l} ={}& -g^{bc}\nabla_{b}Z_{ca_1...a_{k-1}}{}^{m_1...m_l}.
\end{aligned}
\end{equation}
The generalization of \eqref{LdRF}, for any $Z_{ab}{}^{cd}\in \Gamma(\Lambda^{2}\otimes\Lambda^{2*})$, is 
\begin{align}\label{LdRZ}
(\triangle Z)_{ab}{}^{cd} = -\Box Z_{ab}{}^{cd}+2R_{[a}{}^{e}Z_{|e|b]}{}^{cd}-R_{ab}{}^{ef}Z_{ef}{}^{cd}+4R_{[a}{}^{e}{}_{|f}{}^{[c}Z_{e|b]}{}^{|f|d]}.
\end{align}

If $Z\in\Gamma(\Lambda^2\otimes\Lambda^{l*})$, we define the dual $\star Z$ by $(\star Z)_{ab}{}^{m_1...m_l}=\frac{1}{2}\varepsilon_{ab}{}^{cd}Z_{cd}{}^{m_1...m_l}$. This gives a splitting $\Lambda^{2}\otimes\Lambda^{l*}=(\Lambda^{2}_{+}\otimes\Lambda^{l*})\oplus(\Lambda^{2}_{-}\otimes\Lambda^{l*})$, and any $Z\in\Gamma(\Lambda^2\otimes\Lambda^{l*})$ can be written as $Z=Z^{+}+Z^{-}$, with $Z^{\pm}=\frac{1}{2}(Z\pm\star Z)$. 
Similarly to \eqref{orthsplit} and \eqref{LdRFp}, we have, for any $F^{+}\in\Gamma(\Lambda^{2}_{+})$ and $Z\in\Gamma(\Lambda^{2}\otimes\Lambda^{2*})$:
\begin{align}
\la F^{+}\otimes F^{+}, Z \ra_{g} ={}& \la F^{+}\otimes F^{+}, Z^{+} \ra_{g}, \label{orthsplit2} \\
\triangle Z^{+} ={}& 2(dd^{*}Z^{+})^{+}. \label{LdRZp}
\end{align}

Let $W_{abc}{}^{d}$ be the Weyl tensor of $g_{ab}$. Raising an index $W_{ab}{}^{cd}=g^{ce}W_{abe}{}^{d}$, we can write $W=W^{+}+W^{-}$, with $W^{\pm}=\frac{1}{2}(W\pm\star W)$. We shall be interested in the self-dual part $W^{+}$. We have the symmetries $W^{+}{}_{ab}{}^{cd}=W^{+}{}_{[ab]}{}^{cd}=W^{+}{}_{ab}{}^{[cd]}=W^{+cd}{}_{ab}$ and the trace-free property $W^{+}{}_{ab}{}^{ad}=0$. The map $W^{+}{}_{ab}{}^{cd}:\Gamma(\Lambda^{2})\to\Gamma(\Lambda^2)$ is real and symmetric, so it can be diagonalized, and its eigenspaces are orthogonal with respect to \eqref{innprod1}. That is, for $i=1,2,3$, we have 
\begin{align}\label{EVdef}
W^{+}{}_{ab}{}^{cd}F^{i}_{cd} = 2 \lambda_{i} F^{i}_{ab}
\end{align}
(no sum over $i$), where $\lambda_i$ are real and we choose the normalization $\la F^{i},F^{j} \ra_{g}=2\delta^{ij}$. We also have (cf. \cite[Eq. (8.3.8)]{PR2})
\begin{subequations}
\begin{align}
\lambda_1+\lambda_2+\lambda_3 ={}&0, \label{EV1} \\
\lambda_1^2+\lambda_2^2+\lambda_3^2 ={}& \frac{1}{4}|W^{+}|^{2}_{g}. \label{EV2}
\end{align}
\end{subequations}
At any point, the three self-dual 2-forms $F^{i}$ give an orthonormal basis of $\Lambda^2_{+}$. The Levi-Civita connection $\nabla:\Gamma(\Lambda^2_{+})\to\Gamma(\Lambda^1\otimes\Lambda^2_{+})$ can be expressed in terms of a local, matrix-valued connection 1-form $\Gamma^{i}{}_{j}=\Gamma_{a}{}^{i}{}_{j}dx^a$ defined by 
\begin{align}\label{connectionform}
\nabla_{a} F^{i}_{bc} = \sum_{j=1}^{3}\Gamma_{a}{}^{i}{}_{j}F^{j}_{bc}.
\end{align}

\subsection{The main identity}

\begin{lemma}
Let $(M,g)$ be an orientable Riemannian manifold, with non-zero self-dual Weyl tensor $W^{+}$. Let $2\lambda_1,2\lambda_2,2\lambda_3$ be the eigenvalues of $W^{+}$ as in \eqref{EVdef}. Assume $\lambda_3\neq0$, and let $F$ be a corresponding eigenvector normalized as $|F|^2_{g}=2$. Define 
\begin{align}
j^{a} := \nabla_{b}F^{ba}, \qquad V^{a} := F^{ab}j_{b}.
\end{align}
Then the following identity holds:
\begin{align}\label{mainidentity}
\nabla_{a}V^{a} = A+B,
\end{align}
where 
\begin{align}
A ={}& \left| j \right|_{g}^2+\frac{1}{12}\left| \nabla F \right|_{g}^2+\frac{1}{3\lambda_3} \left[ (\lambda_1-\lambda_2)^2 - 2\left( \lambda_1 \left| \Gamma^{3}{}_{1} \right|^2_{g} + \lambda_2 \left| \Gamma^{3}{}_{2} \right|^2_{g} \right) \right], \label{Aterm} \\
B ={}& -\frac{1}{6} \la F\otimes F, dd^{*}(\lambda_3^{-1}W^{+}) \ra_{g}. \label{Bterm}
\end{align}
Here, $\Gamma^{3}{}_{1}$ and $\Gamma^{3}{}_{2}$ are the 1-forms defined in \eqref{connectionform}.
\end{lemma}

\begin{remark}
The identity \eqref{mainidentity} is inspired by ideas due to Wu \cite{Wu} and LeBrun \cite{Lebrun21}. Wu proved the remarkable result that a compact Einstein metric with the property $\det(W^{+})>0$ is necessarily conformally K\"ahler. An alternative proof was given by LeBrun, that can be adapted to manifolds with a different asymptotic structure as in \cite[Proposition 3]{BGL}. Other works using similar Weitzenb\"ock identities to find K\"ahler structures can be found in the literature, e.g. \cite{Micallef}, \cite{Gursky}.
\end{remark}

\begin{proof}
Using that $j_{a}=-(d^{*}F)_{a}$ and $\star F = F$, together with the identities \eqref{orthsplit}, \eqref{LdRFp}, we have:
\begin{align*}
\nabla_{a}V^{a}={}& (\nabla_{a}F^{ab})j_{b} + F^{ab}\nabla_{a}j_{b} \\
={}& j^{a}j_{a} + \frac{1}{2}F^{ab}(d j)_{ab} \\
={}& j^{a}j_{a} - \frac{1}{2}F^{ab}(dd^{*}F)_{ab} \\
={}& |j|_{g}^2-\frac{1}{2}\la F,dd^{*}F\ra_{g} \\
={}& |j|_{g}^2-\frac{1}{2}\la F,(dd^{*}F)^{+}\ra_{g} \\
={}& |j|_{g}^2-\frac{1}{4}\la F,\triangle F\ra_{g}.
\end{align*}
Now, using \eqref{LdRF} and $|F|_{g}^2=2$, a short calculation gives
\[
\la F,\triangle F\ra_{g}=|\nabla F|^2-R_{abcd}F^{ab}F^{cd}+R.
\]
Using that $F$ is self-dual, we have
\[
R_{abcd}F^{ab}F^{cd} = W^{+}_{abcd}F^{ab}F^{cd}+\frac{R}{3}.
\]
Since $F$ is an eigenform of $W^{+}$ with eigenvalue $2\lambda_3$, and since $|F|^{2}_{g}=2$, we get $W^{+}_{abcd}F^{ab}F^{cd}=4\lambda_3$. 
Thus:
\begin{align}\label{div1}
\nabla_{a}V^{a}= |j|_{g}^2-\frac{1}{4}|\nabla F|_{g}^2 + \lambda_3-\frac{R}{6}.
\end{align}
We shall now find a convenient expression for the term $\lambda_3-\frac{R}{6}$. 
Using 
\[
W^{+}{}_{ab}{}^{cd}F_{cd}=2\lambda_3 F_{ab}, \qquad F_{eb}F^{ab}=\frac{1}{2}\delta_{e}^{a}, \qquad R_{abcd}F^{ab}F^{cd} = 4\lambda_3+\frac{R}{3},
\]
together with \eqref{orthsplit2}, \eqref{LdRZp} and \eqref{LdRZ}, we have:
\begin{align*}
2 \la F\otimes F,dd^{*}(\lambda_3^{-1}W^{+}) \ra_{g}
={}& 2 \la F\otimes F, (dd^{*}(\lambda_3^{-1}W^{+}))^{+} \ra_{g} \\
={}& \la F\otimes F, \triangle(\lambda_3^{-1}W^{+}) \ra_{g} \\
={}& -F^{ab}F^{cd}\Box(\lambda_{3}^{-1}W^{+}_{abcd})
+ 2\lambda_3^{-1}F^{ab}F^{cd}R_{a}{}^{e}W^{+}_{ebcd} \\
& -\lambda_3^{-1}F^{ab}F^{cd}R_{ab}{}^{ef}W^{+}_{efcd} 
+4\lambda_3^{-1}F^{ab}F^{cd}R_{aefc}W^{+e}{}_{b}{}^{f}{}_{d} \\
={}& - F^{ab}F^{cd}\Box(\lambda_{3}^{-1}W^{+}_{abcd}) + 4R_{a}{}^{e}F_{eb}F^{ab} \\
& - 2R_{ab}{}^{ef}F^{ab}F_{ef} + 4\lambda_3^{-1}F^{ab}F^{cd}R_{aefc}W^{+e}{}_{b}{}^{f}{}_{d} \\
={}& - F^{ab}F^{cd}\Box(\lambda_{3}^{-1}W^{+}_{abcd}) + 2R \\
& - 2R_{abcd}F^{ab}F^{cd}  + 4\lambda_3^{-1}F^{ab}F^{cd}R_{aefc}W^{+e}{}_{b}{}^{f}{}_{d} \\
={}& - F^{ab}F^{cd}\Box(\lambda_{3}^{-1}W^{+}_{abcd})
+ 8\left(\frac{R}{6}-\lambda_3 \right) \\
& + 4\lambda_3^{-1}F^{ab}F^{cd}R_{aefc}W^{+e}{}_{b}{}^{f}{}_{d}.
\end{align*}
Now, using the decomposition of the Riemann tensor into the Weyl and Ricci tensors and the curvature scalar, we get
\begin{align*}
\lambda_3^{-1}F^{ab}F^{cd}R_{aefc}W^{+e}{}_{b}{}^{f}{}_{d} 
={}& - W^{+}_{abcd}F^{ab}F^{cd}+\frac{R}{6}+\frac{1}{2\lambda_3}|W^{+}|^2_{g} \\
={}& -4\lambda_3 + \frac{R}{6} + \frac{2}{\lambda_3}(\lambda_1^2+\lambda_2^2+\lambda_3^2) \\
={}& \frac{R}{6} - \lambda_3 +\frac{1}{\lambda_3}(\lambda_1-\lambda_2)^2,
\end{align*}
where we used \eqref{EV1} and \eqref{EV2}. On the other hand, the computation of the term $F^{ab}F^{cd}\Box(\lambda_{3}^{-1}W^{+}_{abcd})$ is analogous to the computation performed in the proof of Lemma 3 in \cite{BGL}:
\begin{align*}
F^{ab}F^{cd}\Box(\lambda_{3}^{-1}W^{+}_{abcd}) ={}& 
\Box(\lambda_{3}^{-1}W^{+}_{abcd}F^{ab}F^{cd}) 
- \lambda_3^{-1}W^{+}_{abcd}\Box(F^{ab}F^{cd}) \\
& -2\nabla^{e}(F^{ab}F^{cd})\nabla_{e}(\lambda_{3}^{-1}W^{+}_{abcd}) \\
={}& + \lambda_3^{-1}W^{+}_{abcd}\Box(F^{ab}F^{cd}) 
-2\nabla_{e}\left[ \lambda_{3}^{-1}W^{+}_{abcd}\nabla^{e}(F^{ab}F^{cd})\right] \\
={}& 2\lambda_3^{-1}W^{+}_{abcd} \left( F^{cd}\Box F^{ab}+\nabla_{e}F^{ab}\nabla^{e}F^{cd} \right) \\ 
&- 4\nabla_{e}\left( \lambda_{3}^{-1}W^{+}_{abcd}F^{cd}\nabla^{e}F^{ab}\right) \\
={}& 4F_{ab}\Box F^{ab} + 2\lambda_3^{-1}W^{+}_{abcd}\nabla_{e}F^{ab}\nabla^{e}F^{cd} - 8\nabla_{e}\left(F_{ab}\nabla^{e}F^{ab}\right) \\
={}& -4 \left| \nabla F \right|_{g}^{2} + 2\lambda_3^{-1}W^{+}_{abcd}\nabla_{e}F^{ab}\nabla^{e}F^{cd}
\end{align*}
where we used that $|F|_{g}^2=2$. To find an expression for $W^{+}_{abcd}\nabla_{e}F^{ab}\nabla^{e}F^{cd}$, we use \eqref{connectionform}, which, with $F^{3}\equiv F$, gives
\[
W^{+}_{abcd}\nabla_{e}F^{ab}\nabla^{e}F^{cd} = 4\left( \lambda_1 \left| \Gamma^{3}{}_{1} \right|^2_{g} + \lambda_2 \left| \Gamma^{3}{}_{2} \right|^2_{g} \right).
\]
Thus, we get:
\begin{align*}
2 \la F\otimes F, dd^{*}(\lambda_3^{-1}W^{+}) \ra_{g} ={}& 
+4 \left| \nabla F \right|_{g}^{2} -\frac{8}{\lambda_3}\left( \lambda_1 \left| \Gamma^{1}{}_{3} \right|^2_{g} + \lambda_2 \left| \Gamma^{2}{}_{3} \right|^2_{g} \right) \\
& +12\left(\frac{R}{6}-\lambda_3 \right)+\frac{4}{\lambda_3}(\lambda_1-\lambda_2)^2
\end{align*}
From here we find an expression for $\lambda_3-\frac{R}{6}$. Substituting back into \eqref{div1}, the result \eqref{mainidentity},\eqref{Aterm},\eqref{Bterm} follows.
\end{proof}

\section{Infinitesimal rigidity}
\label{sec:rigidity}

\subsection{Conformal transformations}
\label{sec:conformal}

Let $M$ be a smooth 4-manifold. We shall be interested in two conformally related Riemannian metrics in $M$:
\begin{align}\label{conformal}
\hat{g}_{ab} = (\hat\lambda_3)^2 g_{ab},
\end{align}
where $\hat\lambda_3$ is assumed to be a nonzero eigenvalue of the Weyl tensor of $\hat{g}_{ab}$, that is:
\begin{align}
\hat{W}^{+}{}_{ab}{}^{cd}\hat{F}_{cd} = 2 \hat\lambda_3 \hat{F}_{ab}, 
\qquad 
|\hat{F}|_{\hat{g}}^{2}=2.
\end{align}
The Levi-Civita connection of $\hat{g}_{ab}$ ($g_{ab}$) is $\hat\nabla_{a}$ ($\nabla_{a}$). Indices of hatted (unhatted) quantities are raised and lowered with $\hat{g}_{ab}$ ($g_{ab}$) and its inverse. Since the identity \eqref{mainidentity} is valid for generic Riemannian manifolds, it applies in particular to $\hat{g}_{ab}$:
\begin{equation}\label{mainidentityhat}
\begin{aligned}
& \hat\nabla_{a}\hat{V}^{a} = \hat{A}+\hat{B}, \\
& \hat{A} = | \hat{j} |_{\hat{g}}^2+\frac{1}{12} |\hat{\nabla}\hat{F}|_{\hat{g}}^2 + \frac{1}{3\hat{\lambda}_3} \left[ (\hat{\lambda}_1-\hat{\lambda}_2)^2 - 2\left( \hat{\lambda}_1 |\hat{\Gamma}^{3}{}_{1}|^2_{\hat{g}} + \hat{\lambda}_2 |\hat{\Gamma}^{3}{}_{2}|^2_{\hat{g}} \right) \right],\\
& \hat{B} = -\frac{1}{6}\la\hat{F}\otimes \hat{F}, \hat{d}\hat{d}^{*}(\hat{\lambda}_3^{-1}\hat{W}^{+})\ra_{\hat{g}}. 
\end{aligned}
\end{equation}
Here, $\hat{j}^{a}=\hat\nabla_{b}\hat{F}^{ba}$ and $\hat{V}^{a}=\hat{F}^{ab}\hat{j}_{b}$. The term $\hat{B}$ can be written as 
\begin{align}\label{Bindices}
\hat{B} = \frac{1}{3}\hat{F}^{ab}\hat{F}^{cd}\hat{g}^{ef}\hat\nabla_{a}\hat\nabla_{f}(\hat\lambda_{3}^{-1}\hat{W}^{+}_{ebcd}).
\end{align}
Using the relation between $\nabla_{a}$ and $\hat\nabla_{a}$, we have
\begin{align}\label{conformalrelation}
\hat{g}^{ef}\hat\nabla_{f}(\hat\lambda_{3}^{-1}\hat{W}^{+}_{ebcd}) = 
\hat\lambda_{3}^{-1}g^{ef}\nabla_{f}W^{+}_{ebcd} 
\end{align}
where $W^{+}_{abcd}$ is the self-dual Weyl tensor of $g_{ab}$. Replacing  \eqref{conformalrelation} into \eqref{Bindices} and using Bianchi identities (cf. \cite[Eq. (4.10.7)]{PR1}), after a short calculation we find
\begin{align}\label{Bindices2}
\hat{B} = -\frac{1}{3}\hat{P}^{abcd}\hat\nabla_{a}\left( \hat\lambda_{3}^{-1}\nabla_{d}E_{bc} \right), 
\end{align}
where $E_{bc}=R_{bc}-\frac{R}{4}g_{bc}$ is the trace-free Ricci tensor of $g_{ab}$, and the tensor field $\hat{P}^{abcd}$ is 
\begin{align*}
\hat{P}^{abcd}=\hat{F}^{ab}\hat{F}^{cd}-\frac{1}{6}(\hat{g}^{ac}\hat{g}^{bd}-\hat{g}^{ad}\hat{g}^{bc}+\hat\varepsilon^{abcd}), 
\end{align*}
with $\hat\varepsilon_{abcd}$ the volume form of $\hat{g}_{ab}$. We shall need the expression \eqref{Bindices2} below.

\subsection{ALF curves}

Consider a smooth one-parameter family of metrics $\hat{g}_{ab}(s)$ in $M$. A  tensor field $T(s)$ constructed from the metric has an expansion
\begin{align}\label{analytic}
 T(s) = T+ (\delta{T}) \, s + \frac{1}{2!}(\delta^2{T}) \, s^{2} + O(s^{3}),
\end{align}
where $T,\delta{T},\delta^2{T}, ... $ are tensor fields in $M$,
\begin{align}
\delta^{k} T = \left. \left[\frac{d^{k} }{d s^{k}}T(s)\right] \right|_{s=0}.
\end{align}
We say that $T= T(s)\big|_{s=0}$ is the ``background'' value of $T(s)$, $\delta{T}$ is the linearization of $T(s)$ at $T$, and \eqref{analytic} is a ``perturbative expansion''.

The Levi-Civita connection of $\hat{g}_{ab}(s)$ is denoted $\hat{\nabla}^{(s)}_{a}$, and its relation to the background Levi-Civita connection is (see \cite[Eq. (7.5.7)]{Wald})
\begin{align}\label{nablas}
\hat{\nabla}^{(s)}_{a}u_{b} = \hat{\nabla}_{a}u_{b} - \hat{C}_{ab}{}^{c}(s) \, u_{c}
\end{align}
for all $u_{a}$, where $\hat{\nabla}_{a}=\hat{\nabla}_{a}^{(0)}$ and
\begin{align}\label{difftensor}
\hat{C}_{ab}{}^{c}(s) = \frac{1}{2}\hat{g}^{cd}(s)
\left[ \hat\nabla_{a}\hat{g}_{bd}(s) + \hat\nabla_{b}g_{ad}(s) - \hat\nabla_{d}\hat{g}_{ab}(s) \right].
\end{align}
Note that $\hat{C}_{ab}{}^{c}(s) = s\,\delta{\hat{C}}_{ab}{}^{c} + O(s^2)$. 

We shall now study perturbations of the identity \eqref{mainidentityhat}. A perturbative expansion of the eigenvalue problem of the map $\hat{W}^{+}{}_{ab}{}^{cd}$ is justified by the fact that the latter is real and symmetric, see \cite{Kato}. For the conformal factor in \eqref{conformal}, we choose $\hat\lambda_3$ to be a positive eigenvalue. This is justified by assuming that $g_{ab}(0)$ is Hermitian-Einstein, so there is always a positive eigenvalue, and by continuity this holds for $s$ close to zero.

\begin{lemma}\label{lemma:perturbations}
Let $g_{ab}(s)$ be a smooth curve of Riemannian metrics in $M$, and consider the conformally related metrics $\hat{g}_{ab}(s)$ as in section \ref{sec:conformal}.
Suppose that $g_{ab}=g_{ab}(0)$ is Hermitian-Einstein and that $\frac{d}{d s}g_{ab}(s)|_{s=0}$ is an infinitesimal Einstein deformation, that is
\begin{align}
 \left. E_{ab}(s) \right|_{s=0} = 0, \qquad \left. \frac{d}{d s} E_{ab}(s) \right|_{s=0} = 0, \label{assumptions}
\end{align}
where $E_{ab}(s)$ is the trace-free Ricci tensor of $g_{ab}(s)$. Then the perturbative expansion of the identity \eqref{mainidentityhat} gives the following formulas:
\begin{subequations}\label{PertExpMI}
\begin{align}
& \int_{M}\hat\nabla_{a}\hat{V}^a = \left\{\oint_{\partial M} \hat{F}\wedge\hat{\star}d\:\delta\hat{F}\right\} s + \left\{\oint_{\partial M} \delta^2(\hat{F}\wedge\hat{\star}d \, \hat{F})\right\}\frac{s^2}{2!} + O(s^3), \label{PertExpdV} \\
& \int_{M}\hat{A} = \int_{M}\left\{\frac{|\delta(\hat\nabla\hat{F})|_{\hat{g}}^2}{12} + |\delta\hat{j}|_{\hat{g}}^2 + \frac{(\delta\hat\lambda_1-\delta\hat\lambda_2)^2}{3\hat\lambda_3}+\frac{( |\delta\hat{\Gamma}^{3}{}_{1}|_{\hat{g}}^2 + |\delta\hat{\Gamma}^{3}{}_{2}|_{\hat{g}}^2)}{3} \right\} \frac{s^2}{2!} + O(s^3) \label{PertExpA} \\
& \int_{M}\hat{B} = \left\{-\frac{1}{3}\oint_{\partial M}(\hat\lambda_{3}^{-1}\hat{P}^{abcd}\nabla_{d}\delta^2 E_{bc}) \d\Sigma_a\right\}\frac{s^2}{2!} + O(s^3),\label{PertExpB}
\end{align}
\end{subequations}
\end{lemma}

\begin{proof}
We first show \eqref{PertExpdV}. Recall the general identity $(\hat\nabla_{a}\hat{V}^{a}) \, \hat{\pmb{\varepsilon}} = d(\hat\star\hat{V})$, where $\hat{\pmb{\varepsilon}}$ is the volume form of $\hat{g}$, and on the right side, $\hat{V}$ is interpreted as a 1-form. One can check that $\hat\star\hat{V}=\hat{F}\wedge\hat\star d\hat{F}$. Thus:
\[
(\hat\nabla_{a}\hat{V}^{a}) \, \hat{\pmb{\varepsilon}} = 
d(\hat{F}\wedge\hat\star d\hat{F}).
\]
The perturbative expansion of $\hat{F}\wedge\hat\star d\hat{F}$ is:
\begin{align}\label{PertExp3form}
 (\hat{F}\wedge\hat\star d\hat{F})(s) = 
 \delta(\hat{F}\wedge\hat\star d\hat{F}) \, s + \delta^2(\hat{F}\wedge\hat\star d\hat{F}) \, \frac{s^2}{2!} + O(s^3), 
\end{align}
with
\begin{equation}\label{PertExp3form2}
\begin{aligned}
 \delta(\hat{F}\wedge\hat\star d\hat{F}) ={}& \hat{F}\wedge\hat\star d(\delta\hat{F}), \\
 \delta^2(\hat{F}\wedge\hat\star d\hat{F}) ={}& \hat{F}\wedge\hat\star d(\delta^2\hat{F}) + 2\delta\hat{F}\wedge\hat\star d(\delta\hat{F}) + 2\hat{F}\wedge(\delta\hat\star)d(\delta\hat{F}),
\end{aligned}
\end{equation}
where we used that $d\hat{F}|_{s=0}=0$. This gives \eqref{PertExpdV}.

We now focus on \eqref{PertExpA}. Since $\hat{g}_{ab}=\hat{g}_{ab}(0)$ is K\"ahler, the background self-dual Weyl tensor $\hat{W}^{+}_{abcd}$ has a double eigenvalue: $\hat\lambda_1 = \hat\lambda_2 = -\hat\lambda_3 / 2$, where $\hat\lambda_3>0$.
The parallel background K\"ahler form is $\hat{F}_{ab}$. We then have:
\begin{align*}
\hat{g}_{ab}(s) ={}& \hat{g}_{ab} + \delta\hat{g}_{ab} \, s + O(s^2), \\
\hat{j}_{a}(s) ={}&  \delta\hat{j}_{a} \, s + O(s^2), \\
\hat\nabla^{(s)}\hat{F}(s) ={}& \delta(\hat\nabla\hat{F})\, s + O(s^2), \\
\hat\lambda_1(s) ={}& -\hat\lambda_3 / 2 + \delta\hat\lambda_1 \, s + O(s^2), \\
\hat\lambda_2(s) ={}& -\hat\lambda_3 / 2 + \delta\hat\lambda_2 \, s + O(s^2), \\
\hat{\Gamma}_{a}{}^{3}{}_{1}(s) ={}& \delta\hat{\Gamma}_{a}{}^{3}{}_{1}\, s + O(s^2), \\
\hat{\Gamma}_{a}{}^{3}{}_{2}(s) ={}& \delta\hat{\Gamma}_{a}{}^{3}{}_{2}\, s + O(s^2).
\end{align*}
Replacing into the formula for $\hat{A}$ in \eqref{mainidentityhat}, we get \eqref{PertExpA}.

Finally, we show \eqref{PertExpB}. We shall use the expression \eqref{Bindices2}. Recalling \eqref{assumptions}, we have $\hat{B}(s) = (\delta^{2} \hat{B}) \, \frac{s^2}{2!} + O(s^3)$, and 
\begin{align*}
\delta^{2} \hat{B} ={}& \left. \frac{d^2}{d s^2}\hat{B}(s) \right|_{s=0} \\
={}& -\tfrac{1}{3}\hat{P}^{abcd}(0)\hat\nabla^{(0)}_{a}\left[\hat\lambda_{3}(0)^{-1}\nabla^{(0)}_{d}\delta^2E_{bc}\right] \\
={}& -\tfrac{1}{3}\hat\nabla^{(0)}_{a}\left[\hat\lambda_{3}(0)^{-1}\hat{P}^{abcd}(0)\nabla^{(0)}_{d}\delta^2E_{bc}\right]  
\end{align*}
where $\delta^{2} E_{bc} = \frac{d^2}{d s^2}E_{bc}(s) \big|_{s=0}$, and in the third line we used that $\hat\nabla^{(0)}_{a}\hat{P}^{bcde}(0)=0$. 
Integrating and using the divergence theorem, \eqref{PertExpB} follows.
\end{proof}

\begin{remark}
Since $(\hat\nabla_{a}\hat{V}^a)(s)=\hat{A}(s)+\hat{B}(s)$, and since we see from \eqref{PertExpMI} that the $O(s)$ terms of $\hat{A}(s)$ and $\hat{B}(s)$ vanish, from \eqref{PertExpdV} we deduce that 
\begin{align}\label{order1}
d(\delta\hat{F})=0. 
\end{align}
This was also obtained in \cite{Andersson:2025oae} with different methods.
\end{remark}

We shall now add boundary conditions. We are interested in ALF metrics, cf. \cite[Definition 1.1]{Biquard:2021gwj} for the definition.

\begin{lemma}\label{lemma:linearizedKS}
Under the hypotheses of Lemma \ref{lemma:perturbations}, suppose in addition that the curve $g_{ab}(s)$ is ALF. Then
\begin{align}\label{linearizedKS}
 \left. \frac{d}{d{s}}\left[\hat\nabla_{a}^{(s)}\hat{F}_{bc}(s)\right] \right|_{s=0} = 0.
\end{align}
\end{lemma}

\begin{proof}
The strategy is to show that the $O(s^2)$ boundary terms in $\int_{M}(\hat\nabla_{a}\hat{V}^a)(s)$ and $\int_{M}\hat{B}(s)$ vanish, since then, as every term in the $O(s^2)$ part of the right hand side of \eqref{PertExpA} is positive definite, every such term must vanish as well: in particular $\delta(\hat\nabla{\hat{F}})=0$, which is eq. \eqref{linearizedKS}. 

Consider the $O(s^2)$ boundary term in $\int_{M}(\hat\nabla_{a}\hat{V}^a)(s)$, eq. \eqref{PertExpdV}. From \eqref{PertExp3form}, \eqref{PertExp3form2} and \eqref{order1}, we have $\delta^2(\hat{F}\wedge\hat\star d\hat{F})= \hat{F}\wedge\hat\star d(\delta^2\hat{F})$. In indices:
\begin{align}\label{boundary3form}
[\hat{F}\wedge\hat\star d(\delta^2\hat{F})]_{abc} 
 = -\frac{3}{2}\hat{F}_{[ab}\hat\varepsilon_{c]klm}\hat{g}^{kd}\hat{g}^{le}\hat{g}^{mf}\partial_{d}(\delta^2\hat{F})_{ef}.
\end{align}
We need to integrate this form over a 3-surface $\Sigma_{r}=\{r=const.\}$ (see \cite{Biquard:2021gwj, Aksteiner:2023djq} for the definition of $r$), and then take the limit $r\to\infty$. In the following, all computations will be done with respect to the metric $g_{ab}$. Let $\alpha_{abc}$ denote the 3-form \eqref{boundary3form} restricted to $\Sigma_{r}$, then $\alpha_{abc} = f \mu_{abc}$ for some function $f$, where $\mu_{abc}$ is the volume 3-form on $\Sigma_{r}$ induced from the volume 4-form of $g_{ab}$. We have $|\mu|_{g}=\sqrt{3!}$, hence $|f| = \frac{1}{\sqrt{3!}}|\alpha|_{g}$. To estimate $|\alpha|_{g}$, we first use that $\hat{F}_{ab}=\Omega^{2}F_{ab}$, $\hat\varepsilon_{abcd}=\Omega^{4}\varepsilon_{abcd}$, $\hat{g}^{ab}=\Omega^{-2}g^{ab}$, and that $|F|_{g}=\sqrt{2}$, $|\varepsilon|_{g}=\sqrt{4!}$, $|g|_{g}=2$, then 
\begin{align}\label{norms}
 |\hat{F}|_{g} = \sqrt{2} \, \Omega^2, \qquad 
 |\hat\varepsilon|_{g}=\sqrt{4!} \, \Omega^{4}, \qquad 
 |\hat{g}^{-1}|_{g} = 2 \, \Omega^{-2}.
\end{align}
Here, $\Omega=(\lambda_3)^{1/3}=\hat\lambda_3$. From \cite{Biquard:2021gwj}, the ALF condition implies that 
\begin{align}\label{CFdecay}
 \Omega = \frac{1}{r} + \text{lower order}.
\end{align}
Combining \eqref{norms} and \eqref{CFdecay}, we get $\hat{F}_{ab}=O_{g}(r^{-2})$, $\hat\varepsilon_{abcd}=O_{g}(r^{-4})$, and $\hat{g}^{ab}=O_{g}(r^2)$, where the notation $t_{a_1...a_n}=O_{g}(r^{\alpha})$ means that there is a constant $C$ such that $|t|_{g}\leq C r^{\alpha}$ for sufficiently large $r$, see \cite[Def. 2.6]{Aksteiner:2023djq}. Thus, the factor $\hat{F}_{[ab}\hat\varepsilon_{c]klm}\hat{g}^{kd}\hat{g}^{le}\hat{g}^{mf}$ in \eqref{boundary3form} is $O_{g}(1)$. 
To estimate the remaining factor $\partial_{d}(\delta^2\hat{F})_{ef}$, we first use that $\delta^2\hat{F} = \delta^2(\Omega^2)F + \Omega^2\delta^2 F + 2\delta(\Omega^2)\delta F$, so every term on the right side is $O_{g}(r^{-2})$. In addition, due to the skew-symmetry of \eqref{boundary3form}, $\partial_{d}$ can be replaced by the Levi-Civita connection $\nabla_{d}^{(s)}$of $g_{ab}(s)$, so we can use that by assumption each derivative has extra fall-off, hence $\nabla_{d}^{(s)}(\delta^2\hat{F})_{ef}=O_{g}(r^{-3})$. 

Summarizing, so far we have that the 3-form \eqref{boundary3form} on $\Sigma_{r}$ is $\alpha_{abc} \propto |\alpha|_{g}\mu_{abc}$ and $|\alpha|_{g}=O_{g}(r^{-3})$. From \cite[Def. 1.1]{Biquard:2021gwj}, $\mu_{abc}$ goes like $r^2$. Hence, the integrand in the $O(s^2)$ boundary term in \eqref{PertExpdV} is order $r^{-1}$, thus it vanishes in the limit $r\to\infty$.

The last step is then to analyze the $O(s^2)$ boundary term in $\int_{M}\hat{B}(s)$, eq. \eqref{PertExpB}. To do this, we define 
\[
 \hat{U}^{a} = -\tfrac{1}{3}\hat\lambda_{3}^{-1}\hat{P}^{abcd}\nabla_{d}\delta^2E_{bc}
\]
so that \eqref{PertExpB} becomes
\begin{align}\label{PertExpB2}
\int_{M}\hat{B}(s)  = -\frac{s^2}{2!}\oint_{\partial M} \hat\star\hat{U} + O(s^3).
\end{align}
We have
\begin{align}
(\hat\star\hat{U})_{abc} = \frac{1}{3}\hat\lambda_{3}^{-1}\hat\varepsilon_{abcd}\hat{g}^{dk}\hat{g}^{el}\hat{g}^{fm}\hat{g}^{gn}\hat{P}_{klmn}\nabla_{e}\delta^2E_{fg}
\label{*U}
\end{align}
Using that the ALF assumption gives 
$\hat\lambda_{3}^{-1} = O_{g}(r)$, 
$\hat\varepsilon_{abcd} = O_{g}(r^{-4})$, 
$\hat{g}^{dk} = O_{g}(r^{2})$,  
$\hat{P}_{klmn} = O_{g}(r^{-4})$, 
we find that the term $\hat\lambda_{3}^{-1}\hat\epsilon_{abcd}\hat{g}^{dk}\hat{g}^{el}\hat{g}^{fm}\hat{g}^{gn}\hat{P}_{klmn}$ in \eqref{*U} is $O_{g}(r)$. To estimate the remaining factor $\nabla_{e}\delta^2E_{fg}$, we first recall the discussion around \eqref{nablas}, which gives
\[
\nabla_{e}\delta^2E_{fg} = \nabla^{(s)}_{e}\delta^2E_{fg} + O(s).
\]
Therefore, we can now use that $\delta^2E_{fg} = O_{g}(r^{-3})$ and
$\nabla^{(s)}_{e}\delta^2E_{fg} = O_{g}(r^{-4})$, 
where we used that by assumption each $\nabla^{(s)}_{a}$-derivative has extra fall-off. Recalling \eqref{*U}, we then get $(\hat\star\hat{U})_{abc} = O_{g}(r^{-3})$.

Taking now a 3-surface $\Sigma_{r}=\{r=const.\}$, we have $\oint_{\partial M} \hat\star\hat{U}=\lim_{r\to\infty}\oint_{\Sigma_{r}}\hat\star\hat{U}$. Since, as argued before, the induced volume form is $\mu_{abc}=O_{g}(r^2)$, then $\oint_{\Sigma_{r}}\hat\star\hat{U}$ is order $r^{-1}$. Thus, taking the limit $r\to\infty$ we get $\oint_{\partial M} \hat\star\hat{U} = 0$.

In summary, we showed that, for ALF vacuum perturbations of ALF instantons, the $O(s^2)$ boundary terms of $\int_{M}(\hat\nabla_{a}\hat{V}^a)(s)$ and $\int_{M}\hat{B}(s)$ vanish. Hence, \eqref{linearizedKS} holds for ALF instantons.
\end{proof}

\subsection{Proof of Theorem \ref{maintheorem-early}}
\label{sec:proofmaintheorem}

Let $g_{ab}(s)$ be a curve of Riemannian ALF metrics on $M$, where $g_{ab}(0)=g_{ab}$ is Ricci-flat and Hermitian non-K\"ahler, and $\delta g_{ab}= \frac{d}{d{s}}g_{ab}(s)|_{s=0}$ is an infinitesimal Einstein deformation.
The tensor field $J^{a}{}_{b}(s):=\sqrt{2}\:\hat{F}_{bc}(s)\hat{g}^{ca}(s)$ is an almost-complex structure, and Lemma \ref{lemma:linearizedKS} implies that 
\begin{align}
\hat\nabla_{a}^{(s)}J^{b}{}_{c}(s)=O(s^2),
\end{align}
i.e. it is parallel to second order.
It follows that the curve $g_{ab}(s)$ is conformally K\"ahler (in particular, Hermitian) to second order. Hence the first order perturbation $\delta{g}_{ab}$ solves the linearized Hermitian-Einstein condition.
Since Hermitian, ALF instantons have been classified \cite{Biquard:2021gwj, Li:2023jlq} and the moduli space is smooth and finite-dimensional, and since the topology of $M$ is fixed, it follows that $\delta{g}_{ab}$ must be a perturbation with respect to the moduli parameters, up to gauge. 
Thus, any ALF Ricci-flat perturbation is tangent to the moduli space of Hermitian metrics, in other words, infinitesimal rigidity and integrability hold.

\subsection*{Acknowledgements}
We are grateful to Claude LeBrun for clarifying remarks.
The work of BA is supported by the ERC Consolidator/UKRI Frontier grant TwistorQFT EP/Z000157/1. The work of LA is partially supported by the National Natural Science Foundation of China, under Grant Number W2431012.

\end{document}